\documentclass[11pt,reqno]{amsart}
\usepackage{a4wide}
\usepackage[utf8x]{inputenc}
\usepackage[T1]{fontenc}
\usepackage[english]{babel}
\usepackage{hyperref}
\usepackage{amsfonts,amsmath,amssymb,amsthm,amsrefs}
\usepackage{latexsym}
\usepackage{layout}
\usepackage{dsfont}
\usepackage{color}

\newtheorem{theorem}{Theorem}[section]
\newtheorem{proposition}{Proposition}[section]
\newtheorem{lemma}{Lemma}[section]

\newcommand{\p}{\mathbb{P}}
\newcommand{\e}{\mathbb{E}}

\newcommand{\reals}{\mathbb{R}}
\newcommand{\ind}{\mathbf{1}}

\newcommand{\ruintime}{\sigma^\pi_p}
\newcommand{\ruinbarrier}{\sigma^b_p}

\allowdisplaybreaks

\begin{document}

\title[De Finetti's control problem with Parisian ruin]{De Finetti's control problem with Parisian ruin for spectrally negative Lévy processes}

\author[]{Jean-Fran\c{c}ois Renaud}
\address{D\'epartement de math\'ematiques, Universit\'e du Qu\'ebec \`a Montr\'eal (UQAM), 201 av.\ Pr\'esident-Kennedy, Montr\'eal (Qu\'ebec) H2X 3Y7, Canada}
\email{renaud.jf@uqam.ca}

\date{\today}


\begin{abstract}
We consider de Finetti's stochastic control problem when the (controlled) process is allowed to spend time under the critical level. More precisely, we consider a generalized version of this control problem in a spectrally negative Lévy model with exponential Parisian ruin. We show that, under mild assumptions on the Lévy measure, an optimal strategy is formed by a barrier strategy and that this optimal barrier level is always less than the optimal barrier level when classical ruin is implemented. Also, we give necessary and sufficient conditions for the barrier strategy at level zero to be optimal.
\end{abstract}

\maketitle

\section{Introduction and main result}

On a filtered probability space $\left( \Omega, \mathcal{F}, \mathbb{F}, \p \right)$, let $X=\left\lbrace X_t , t \geq 0 \right\rbrace$ be a spectrally negative Lévy process with Laplace exponent $\theta \mapsto \psi (\theta)$ and with $q$-scale functions $\left\lbrace W^{(q)}, q \geq 0 \right\rbrace$ given by
$$
\int_0^\infty \mathrm{e}^{-\theta x} W^{(q)}(x) \mathrm{d}x = \left(\psi(\theta)-q \right)^{-1} ,
$$
for all $\theta> \Phi(q)=\sup \left\lbrace \lambda \geq 0 \colon \psi(\lambda)=q \right\rbrace$. Recall that
$$
\psi(\theta) = \gamma \theta + \frac{1}{2} \sigma^2 \theta^2 + \int^{\infty}_0 \left( \mathrm{e}^{-\theta z} - 1 + \theta z \ind_{(0,1]}(z) \right) \nu(\mathrm{d}z) ,
$$
where $\gamma \in \reals$ and $\sigma \geq 0$, and where $\nu$ is a $\sigma$-finite measure on $(0,\infty)$, called the L\'{e}vy measure of $X$, satisfying
\begin{equation*}
\int^{\infty}_0 (1 \wedge x^2) \nu(\mathrm{d}x) < \infty .
\end{equation*}
For more details on spectrally negative Lévy processes and scale functions, see e.g.\ \cite{kuznetsov-e-tal_2012,kyprianou_2014}.

In what follows, we will use the following notation: the law of $X$ when starting from $X_0 = x$ is denoted by $\p_x$ and the corresponding expectation by $\e_x$. We write $\p$ and $\e$ when $x=0$.

\subsection{Problem formulation}

Let the spectrally negative Lévy process $X$ be the underlying surplus process. A dividend strategy $\pi$ is represented by a non-decreasing, left-continuous and adapted stochastic process $L^\pi = \left\lbrace L^\pi_t , t \geq 0 \right\rbrace$, where $L^\pi_t$ represents the cumulative amount of dividends paid up to time $t$ under this strategy, and such that $L^\pi_0 = 0$. For a given strategy $\pi$, the corresponding controlled surplus process $U^\pi = \left\lbrace U^\pi_t , t \geq 0 \right\rbrace$ is defined by $U^\pi_t = X_t - L^\pi_t$. The stochastic control problem considered in this paper involves the time of Parisian ruin (with rate $p>0$) for $U^\pi$ defined by
$$
\ruintime = \inf \left\lbrace t>0 \colon t-g_t^\pi > \mathbf{e}_p^{g_t^\pi} \; \text{and} \; U^\pi_t < 0 \right\rbrace ,
$$
where $g_t^\pi = \sup \left\lbrace 0 \leq s \leq t \colon U^\pi_s \geq 0 \right\rbrace$, with $\mathbf{e}_p^{g_t^\pi}$ an independent random variable, following the exponential distribution with mean $1/p$, associated to the corresponding excursion below $0$ (see \cite{baurdoux-et-al_2016} for more details). Note that, without loss of generality, we have chosen $0$ to be the critical level.

A strategy $\pi$ is said to be admissible if a dividend payment is not larger than the current surplus level, i.e.\ $L^\pi_{t+} - L^\pi_t \leq U^\pi_t$, for all $t < \ruintime$, and if no dividends are paid when the controlled surplus is negative, i.e.\ $t \mapsto L^\pi_t \, \ind_{(-\infty,0)} (U^\pi_t) \equiv 0$. The set of admissible dividend strategies will be denoted by $\Pi_p$.

Fix a discounting rate $q \geq 0$. The value function associated to an admissible dividend strategy $\pi \in \Pi_p$ is defined by
$$
v_\pi (x) = \e_x \left[ \int_0^{\ruintime} \mathrm{e}^{-q t} \mathrm{d}L^\pi_t \right] , \quad x \in \reals .
$$
In particular, for $\pi \in \Pi_p$ and $x<0$, using the strong Markov property and the spectral negativity of $X$, we can easily verify that
\begin{equation}\label{eq:value-under-zero}
v_\pi(x)=\e_x \left[\mathrm{e}^{-q \tau_0^+} \ind_{\{\tau_0^+ < \mathbf{e}_p\}} \right] v_\pi (0) = \mathrm{e}^{\Phi(p+q) x} v_\pi (0) ,
\end{equation}
where $\tau_0^+ = \inf \left\lbrace t>0 \colon X_t >0 \right\rbrace$ and where $\mathbf{e}_p$ is an independent exponentially distributed random variable with mean $1/p$, thanks to the well-known fluctuation identity (see e.g.\ \cite{kyprianou_2014})
\begin{equation}\label{eq:first-passage-above}
\e_x \left[\mathrm{e}^{-r \tau_b^+} \ind_{\{\tau_b^+ < \infty\}} \right] = \mathrm{e}^{-\Phi(r) (b-x)} , \; x \leq b ,
\end{equation}
where $\tau_b^+ = \inf \left\lbrace t>0 \colon X_t >b \right\rbrace$.

The goal is to find the optimal value function $v_\ast$ defined by
$$
v_\ast (x) = \sup_{\pi \in \Pi_p} v_\pi (x)
$$
and an optimal strategy $\pi_\ast \in \Pi_p$ such that
$$
v_{\pi_\ast} (x) = v_\ast (x) ,
$$
for all $x \in \reals$. Because of the Parisian nature of the time of ruin considered in this control problem, we have to deal with possibly negative starting capital.

\subsection{Main result and organization of the paper}

Let us introduce the family of horizontal barrier strategies, also called reflection strategies. For $b \in \reals$, the (horizontal) barrier strategy at level $b$ is the strategy denoted by $\pi^b$ and with cumulative amount of dividends paid until time $t$ given by $L_t^b = \left( \sup_{0<s\leq t} X_s - b \right)_+$, for $t>0$. If $X_0 = x > b$, then $L^b_{0+} = x-b$. Note that, if $b \geq 0$, then $\pi_b \in \Pi_p$. The corresponding value function is thus given by
$$
v_b (x) = \e_x \left[ \int_0^{\ruinbarrier} \mathrm{e}^{-q t} \mathrm{d}L_t^b \right] ,
$$
for all $x \in \reals$, where $\ruinbarrier$ is the time of Parisian ruin (with rate $p>0$) for the controlled process $U^b_t = X_t - L_t^b$.

Before stating the main result of this paper, recall that the tail of the L\'evy measure is the function $x\mapsto\nu(x,\infty)$, where $x\in(0,\infty)$, and that a function $f \colon (0,\infty) \to (0,\infty)$ is log-convex if the function $\log(f)$ is convex on $(0,\infty)$.

\begin{theorem}\label{T:main}
Fix $q \geq 0$ and $p>0$. If the tail of the L\'evy measure is log-convex, then an optimal strategy for the control problem is formed by a barrier strategy.
\end{theorem}

The original version of de Finetti's optimal dividends problem, i.e.\ when the time of ruin is the first passage time below the critical level (intuitively, when $p \to \infty$), has been extensively studied. In a spectrally negative Lévy model, following the work of Avram, Palmowski \& Pistorius \cite{avram-et-al_2007}, an important breakthrough was made by Loeffen \cite{loeffen_2008}; in the latter, a sufficient condition, on the Lévy measure $\nu$, is given for a barrier strategy to be optimal. This condition was further relaxed by Loeffen \& Renaud \cite{loeffen-renaud_2010}; in this other paper, it is shown that if the tail of the Lévy measure is log-convex then a barrier strategy is optimal for de Finetti's optimal dividends problem with an affine penalty function at ruin (if we set $S=K=0$ in that paper, we recover the classical problem). To the best of our knowledge, this still stands as the mildest condition for the optimality of a barrier strategy in a spectrally negative Lévy model. Finally, note that Czarna \& Palmowski \cite{czarna-palmowski_2014} have considered de Finetti's control problem with deterministic Parisian delays.

The rest of the paper is organized as follows. First, we compute the value function of an arbitrary horizontal barrier strategy and then find the optimal barrier level $b^\ast_p$ (see the definition in~\eqref{eq:optimal-barrier-level}). Second, we derive the appropriate verification lemma for this control problem and prove that the barrier strategy at level $b^\ast_p$ is optimal.

\section{Horizontal barrier strategies}

Before computing the value function of an arbitrary barrier strategy at level $b$, we have to define another family of scale functions, also called second $q$-scale functions of $X$.

\subsection{Second family of scale functions}

For each $x \in \reals$ and $q,\theta \geq 0$, define as in \cite{albrecher-et-al_2016}
\begin{equation}\label{eq:Zq-def}
Z_q (x,\theta) = \mathrm{e}^{\theta x} \left( 1 - (\psi (\theta)-q) \int_0^x \mathrm{e}^{-\theta y} W^{(q)}(y) \mathrm{d}y \right) .
\end{equation}
Note that, for $x \leq 0$ or for $\theta=\Phi(q)$, we have $Z_q (x,\theta)=\mathrm{e}^{\theta x}$. In what follows, $Z_q^\prime (x,\theta)$ will represent the derivative with respect to the first argument. Consequently, for $x>0$, we have $Z_q^\prime (x,\theta) = \theta Z_q (x,\theta) - (\psi (\theta)-q) W^{(q)}(x)$ and, for $x<0$, we have $Z_q^\prime (x,\theta)=\theta \mathrm{e}^{\theta x}$. Note that $Z_q^\prime (0-,\theta)=Z_q^\prime (0+,\theta)=\theta$ if and only if $W^{(q)}(0)=0$, i.e.\ if and only if $X$ has paths of unbounded variation.

In this paper, we will encounter the function $Z_q$ when $\theta=\Phi(p+q)$, that is the function
$$
Z_q (x,\Phi(p+q)) = \mathrm{e}^{\Phi(p+q) x} \left( 1 - p \int_0^x \mathrm{e}^{-\Phi(p+q) y} W^{(q)}(y) \mathrm{d}y \right) ,
$$
from which we deduce that, for $x>0$,
\begin{equation}\label{eq:Zqprime-def}
Z_q^\prime (x,\Phi(p+q))=\Phi(p+q)Z_q (x,\Phi(p+q))-pW^{(q)}(x) .
\end{equation}
Consequently, set $Z_q^\prime (0,\Phi(p+q))=\Phi(p+q)-pW^{(q)}(0+)$. Since we assume that $p>0$, we have that $\Phi(p+q)>\Phi(q)$ and we can write
\begin{equation}\label{eq:Zq-second-rep}
Z_q (x,\Phi(p+q)) = p \int_0^\infty \mathrm{e}^{-\Phi(p+q) y} W^{(q)}(x+y) \mathrm{d}y .
\end{equation}
Then, for $x > 0$, we have
\begin{equation}\label{eq:Zqprime-second-rep}
Z_q^\prime (x,\Phi(p+q)) = p \int_0^\infty \mathrm{e}^{-\Phi(p+q) y} W^{(q)\prime}(x+y) \mathrm{d}y ,
\end{equation}
which is well defined since $W^{(q)}$ is differentiable almost everywhere (see e.g.\ Lemma 2.3 in \cite{kuznetsov-e-tal_2012}). Clearly, $x \mapsto Z_q (x,\Phi(p+q))$ is a non-decreasing continuous function.

\subsection{Value function of a barrier strategy}

Here is the value of an arbitrary admissible barrier strategy:
\begin{proposition}\label{P:value-barrier}
For $q,b \geq 0$, the value function associated to $\pi_b$ is given by
\begin{equation}\label{eq:value-barrier}
v_b (x) =
\begin{cases}
\frac{Z_q (x,\Phi(p+q))}{Z_q^\prime (b,\Phi(p+q))} & \text{for $x \in (-\infty,b]$,}\\
x-b+v_b (b) & \text{for $x \in (b,\infty)$.}
\end{cases}
\end{equation}
\end{proposition}
\begin{proof}
See the proof in Appendix~\ref{A:value-barrier}.
\end{proof}


Using Parisian ruin with rate $p$ allows to fill in the spectrum of possibilities between classical ruin (no delay, $p \to \infty$) and no ruin at all (infinite delays, $p \to 0$). To illustrate this, let us have a look at the behaviour of $v_b$, which depends implicitly on the Parisian rate $p$, when:
\begin{enumerate}
\item $p \to 0$ (infinite delays, no ruin);
\item $p \to \infty$ (no delay, classical ruin).
\end{enumerate}

First, it is known that (see e.g.\ Equation~(3.15) in \cite{avram-et-al_2007}), if there is no ruin, then
$$
\e_b \left[ \int_0^{\infty} \mathrm{e}^{-q t} \mathrm{d}L_t^b \right] = \frac{1}{\Phi(q)} .
$$
Using the strong Markov property and the spectral negativity of $X$, for $-\infty < x \leq b$, we thus have
$$
\e_x \left[ \int_0^{\infty} \mathrm{e}^{-q t} \mathrm{d}L_t^b \right] = \frac{\mathrm{e}^{-\Phi(q) (b-x)}}{\Phi(q)} ,
$$
where we used again the identity in~\eqref{eq:first-passage-above}.

On the other hand, for $-\infty < x \leq b$, using Equation~\eqref{eq:value-barrier}, we get
$$
v_b(x) = \frac{Z_q (x,\Phi(p+q))}{Z_q^\prime (b,\Phi(p+q))} \underset{p \to 0}{\longrightarrow} \frac{Z_q (x,\Phi(q))}{Z_q^\prime (b,\Phi(q))} = \frac{\mathrm{e}^{\Phi(q) x}}{\Phi(q) \mathrm{e}^{\Phi(q) b}} = \frac{1}{\Phi(q)} \mathrm{e}^{-\Phi(q) (b-x)} .
$$


Second, recall that the value of a barrier strategy at level $b$ subject to classical ruin is given by
$$
\e_b \left[ \int_0^{\sigma^b_\infty} \mathrm{e}^{-q t} \mathrm{d}L_t^b \right] = \frac{W^{(q)}(x)}{W^{(q)\prime}(b)} , \quad x \leq b ,
$$
where $\sigma^b_\infty:=\inf \left\lbrace t>0 \colon U_t^b < 0 \right\rbrace$; see e.g.\ \cite{avram-et-al_2007}. Since $W^{(q)}(x)=0$ if $x<0$, this last identity is valid for all $-\infty< x \leq b$. On the other hand, for $-\infty < x \leq b$, using Equation~\eqref{eq:value-barrier} and the fact that $\lim_{p \to \infty} \Phi(p+q) = \infty$, we get
$$
v_b(x) = \frac{Z_q (x,\Phi(p+q))}{Z_q^\prime (b,\Phi(p+q))} = \frac{\Phi(p+q) \int_0^\infty \mathrm{e}^{-\Phi(p+q) y} W^{(q)}(x+y) \mathrm{d}y}{\Phi(p+q) \int_0^\infty \mathrm{e}^{-\Phi(p+q) y} W^{(q)\prime}(b+y) \mathrm{d}y} \underset{p \to \infty}{\longrightarrow} \frac{W^{(q)}(x)}{W^{(q)\prime}(b)} ,
$$
where, in the last step, we used the Initial Value Theorem for Laplace transforms (see e.g.\ \cite{feller_1971}). Note that, when $x<0$, we have directly
$$
Z_q (x,\Phi(p+q))=\mathrm{e}^{\Phi(p+q) x} \underset{p \to \infty}{\longrightarrow} 0 = W^{(q)}(x) .
$$

\subsection{Optimal barrier level}

As defined in \cite{loeffen_2008,loeffen-renaud_2010}, the optimal barrier level in de Finetti's classical control problem is given by
$$
b_\infty^\ast = \sup \left\lbrace b \geq 0 \colon W^{(q)\prime} (b) \leq W^{(q)\prime} (x) , \; \text{for all $x \geq 0$} \right\rbrace .
$$
Similarly, let us define the candidate for the optimal barrier level for the current version of this control problem by
\begin{equation}\label{eq:optimal-barrier-level}
b_p^\ast = \sup \left\lbrace b \geq 0 \colon Z_q^\prime (b,\Phi(p+q)) \leq Z_q^\prime (x,\Phi(p+q)) , \; \text{for all $x \geq 0$} \right\rbrace .
\end{equation}

\begin{proposition}\label{P:optimal-barrier-level}
Fix $q \geq 0$ and $p>0$. Suppose the tail of the Lévy measure is log-convex.
Then, we have that $0 \leq b_p^\ast \leq b_\infty^\ast$. Further, $b_p^\ast > 0$ if and only if one of the following three cases hold:
\begin{enumerate}
\item[(a)] $\sigma>0$ and $\left(\Phi(p+q) \right)^2/p < 2/\sigma^2$;
\item[(b)] $\sigma=0$ and $\nu(0,\infty)=\infty$;
\item[(c)] $\sigma=0$, $\nu(0,\infty)<\infty$ and
$$
\frac{c \Phi(p+q)}{p} \left( \Phi(p+q) - \frac{p}{c} \right) < \frac{q+\nu(0,\infty)}{c} ,
$$
where $c=\gamma+\int_0^1 x \nu(\mathrm{d}x)$.
\end{enumerate}
\end{proposition}

\begin{proof}
See the proof in Appendix~\ref{A:proof-optimal-barrier}.
\end{proof}

First of all, note from Proposition~\ref{P:optimal-barrier-level} that the optimal barrier level $b_p^\ast$, when Parisian ruin with rate $p$ is implemented, is always lower than the optimal barrier level $b_\infty^\ast$ when classical ruin is used.

Also, from (a) and (c), when $\sigma>0$ or $\nu(0,\infty)<\infty$, we have that if the Parisian rate $p>0$ is small enough (large delays), then the barrier strategy at level zero is optimal. If $X_t=ct+\sigma B_t$ is a Brownian motion with drift, then
$$
\Phi(p+q)= \frac{1}{\sigma^2} \left(\sqrt{c^2+2 \sigma^2 (p+q)}-c \right) .
$$
In this case, using (a), we can verify that if the Brownian coefficient $\sigma$ is large enough, then the barrier strategy at level zero is optimal.

Interestingly, in (c), the condition can be re-written as follows:
$$
\frac{c \Phi(p+q)}{p} \; \e \left[ \int_0^{\sigma_\infty^0} \mathrm{e}^{-q t} \mathrm{d}L^0_t \right] < \e \left[ \int_0^{\sigma_p^0} \mathrm{e}^{-q t} \mathrm{d}L^0_t \right] = v_0(0) .
$$
See e.g.\ Equation (3.14) in \cite{avram-et-al_2007}.

\section{Verification lemma and proof of the main result}

Define the operator $\Gamma$ associated with $X$ by
\begin{equation}\label{eq:generator}
\Gamma v(x) = \gamma v^\prime (x)+\frac{\sigma^2}{2} v^{\prime \prime}(x) + \int_{0}^\infty \left( v(x-z)-v(x)+v'(x)z\ind_{(0,1]}(z) \right) \nu(\mathrm{d}z) ,
\end{equation}
where $v$ is a function defined on $\reals$ such that $\Gamma v(x)$ is well defined.

Next is the verification lemma of our stochastic control problem. As the controlled process is now allowed to spend time below the critical level, it is different from the classical verification lemma (see \cite{loeffen_2008}).

\begin{lemma}\label{verificationlemma}
Let $\Gamma$ be the operator defined in~\eqref{eq:generator}. Suppose that $\hat{\pi} \in \Pi_p$ is such that $v_{\hat{\pi}}$ is sufficiently smooth and that, for all $x \in \reals$,
$$
\left(\Gamma-q-p\ind_{(-\infty,0)} \right) v_{\hat{\pi}} (x) \leq 0
$$
and, for all $x>0$, $v^\prime_{\hat{\pi}}(x) \geq 1$. In this case, $\hat{\pi}$ is an optimal strategy for the control problem.
\end{lemma}
\begin{proof}
Set $w:=v_{\hat{\pi}}$ and let $\pi \in \Pi_p$ be an arbitrary admissible strategy. As $w$ is sufficiently smooth, applying an appropriate change-of-variable/version of Ito's formula to the joint process $\left( t, \int_0^t \ind_{(-\infty,0)}(U^\pi_r) \mathrm{d}r , U^\pi_t \right)$ yields
\begin{multline*}
\mathrm{e}^{-q t - p \int_0^t \ind_{(-\infty,0)}(U^\pi_r) \mathrm{d}r} w \left(U^\pi_t \right) - w \left(U^\pi_0 \right) \\
= \int_0^t \mathrm{e}^{-q s - p \int_0^s \ind_{(-\infty,0)}(U^\pi_r) \mathrm{d}r} \left[ (\Gamma-q) w \left(U^\pi_s \right) - p \ind_{(-\infty,0)} \left(U^\pi_s \right) w \left(U^\pi_s \right) \right] \mathrm{d}s \\
- \int_0^t \mathrm{e}^{-q s - p \int_0^s \ind_{(-\infty,0)}(U^\pi_r) \mathrm{d}r} w^\prime \left(U^\pi_{s-} \right) \mathrm{d}L^\pi_s + M^\pi_t \\
+ \sum_{0 < s \leq t} \mathrm{e}^{-q s - p \int_0^s \ind_{(-\infty,0)}(U^\pi_r) \mathrm{d}r} \left[  w \left(U^\pi_{s-} - \Delta L^\pi_s \right) - w \left(U^\pi_{s-} \right) + w^\prime \left(U^\pi_{s-} \right) \Delta L^\pi_s \right] ,
\end{multline*}
where $M^\pi=\left\lbrace M^\pi_t , t \geq 0 \right\rbrace$ is a (local) martingale.

Let $\mathcal{F}_\infty$ denote the sigma-field generated by the whole trajectories of our processes. Consider an independent (of $\mathcal{F}_\infty$) Poisson process with intensity measure $p \, \mathrm{d}t$ and jump times $\left\lbrace T^p_i , i \geq 1 \right\rbrace$. Therefore, we can write
$$
\mathrm{e}^{- p \int_0^s \ind_{(-\infty,0)}(U^\pi_r) \mathrm{d}r} = \p_x \left( T^p_i \notin \left\lbrace r \in (0,s] \colon U^\pi_r < 0 \right\rbrace , \; \text{for all $i \geq 1$} \vert \mathcal{F}_\infty \right) = \e_x \left[ \ind_{\{\sigma_p^\pi > s\}} \vert \mathcal{F}_\infty \right]
$$
and consequently
\begin{multline*}
\e_x \left[ \int_0^t \mathrm{e}^{-q s - p \int_0^s \ind_{(-\infty,0)}(U^\pi_r) \mathrm{d}r} \mathrm{d}L^\pi_s \right] = \e_x \left[ \int_0^t \mathrm{e}^{-q s} \e_x \left[ \ind_{\{\sigma_p^\pi > s\}} \vert \mathcal{F}_\infty \right] \mathrm{d}L^\pi_s \right] \\
= \e_x \left[ \int_0^{\sigma_p^\pi \wedge t} \mathrm{e}^{-q s} \mathrm{d}L^\pi_s \right] ,
\end{multline*}
where we used the definition of a Riemann-Stieltjes integral and the monotone convergence theorem for conditional expectations.

Now, as for all $x \in \reals$,
$$
\left(\Gamma-q-p\ind_{(-\infty,0)} \right) w (x) \leq 0
$$
and, for all $x>0$, $w^\prime(x) \geq 1$, using standard arguments (see e.g.\ \cite{loeffen_2008}) and our definition of an admissible strategy, e.g.\ that $L^\pi$ is identically zero when $U^\pi$ is below zero, we get
$$
w(x) \geq \e_x \left[ \int_0^\infty \mathrm{e}^{-q s - p \int_0^s \ind_{(-\infty,0)}(U^\pi_r) \mathrm{d}r} \mathrm{d}L^\pi_s \right] = \e_x \left[ \int_0^{\sigma_p^\pi} \mathrm{e}^{-q s} \mathrm{d}L^\pi_s \right]=v_\pi(x) .
$$
This concludes the proof.
\end{proof}

The rest of this section is devoted to proving Theorem~\ref{T:main}, i.e.\ proving that an optimal strategy for the control problem is formed by the barrier strategy at level $b^\ast:=b_p^\ast$.

By the definition of $b^\ast$ given in~\eqref{eq:optimal-barrier-level}, for $0 \leq x \leq b^\ast$, we have
$$
v_{b^\ast}^\prime (x) = \frac{Z_q^\prime (x,\Phi(p+q))}{Z_q^\prime (b^\ast,\Phi(p+q))} \geq 1 .
$$
By the definition of $v_{b^\ast}$, for $x > b^\ast$, we have $v_{b^\ast}^\prime (x)=1$. This means $v_{b^\ast}^\prime (x) \geq 1$, for all $x \geq 0$.

Note that, for any $x \in \reals$, we have
\begin{align*}
\left( \Gamma - q-p \right) \mathrm{e}^{\Phi(p+q) x} &= \mathrm{e}^{\Phi(p+q) x} \left( \gamma \Phi(p+q) +\frac{\sigma^2}{2} \Phi^2(p+q) \right) \\
& \qquad + \mathrm{e}^{\Phi(p+q) x} \left[ \int_{0}^\infty \left( \mathrm{e}^{-\Phi(p+q) z}-1+\Phi(p+q) z\ind_{(0,1]}(z) \right) \nu(\mathrm{d}z) - (q+p) \right] \\
&= \mathrm{e}^{\Phi(p+q) x} \left[ \psi \left( \Phi(p+q) \right) - (q+p)\right] = 0 .
\end{align*}
Consequently, for $x<0$, we have
$$
\left(\Gamma-q-p\right) Z_q (x,\Phi(p+q)) = 0
$$
and, for $x \geq 0$, using~\eqref{eq:Zq-second-rep}, we have
$$
\left(\Gamma-q\right) Z_q (x,\Phi(p+q)) = p \int_0^\infty \mathrm{e}^{\Phi(p+q) y} \left(\Gamma-q\right) W^{(q)} (x+y) \mathrm{d}y = 0 ,
$$
since $\left(\Gamma-q\right) W^{(q)} (x) = 0$ for all $x>0$ (see e.g.\ \cite{biffis-kyprianou_2010}). As a consequence, and since $v_{b^\ast}$ is smooth in $x=b^\ast$, we have
$$
\left(\Gamma-q-p\ind_{(-\infty,0)} \right) v_{b^\ast} (x)=0 , \quad \text{for $x \leq b^\ast$.}
$$
All is left to verify now is that $\left(\Gamma-q \right) v_{b^\ast} (x) \leq 0$, for all $x > b^\ast$. It can be done following the same steps as in the proof of Theorem 2 in \cite{loeffen_2008}, thanks to the smoothness of the scale function $Z_q (\cdot, \Phi(p+q))$. The details are left to the reader.

\section*{Acknowledgements}

Funding in support of this work was provided by the Natural Sciences and Engineering Research Council of Canada (NSERC).

%
%
\bibliographystyle{abbrv}
\bibliography{references-SNLPs.bib}

\appendix

\section{Proof of Proposition~\ref{P:value-barrier}}\label{A:value-barrier}

To prove this result, we adapt the methodology used in \cite{renaud-zhou_2007}; see also Equation (15) in~\cite{albrecher-ivanovs_2014}. Let us define $\kappa^p$ as the time of Parisian ruin with rate $p$ for $X$ or, said differently, the time of Parisian ruin when the pay-no-dividend strategy, i.e.\ the strategy $\pi$ with $L_t^\pi \equiv 0$, is implemented. More precisely, define
$$
\kappa^p = \inf \left\lbrace t>0 \colon t-g_t > \mathbf{e}_p^{g_t} \; \text{and} \; X_t < 0 \right\rbrace ,
$$
where $g_t = \sup \left\lbrace 0 \leq s \leq t \colon X_s \geq 0 \right\rbrace$. Let us also define, for $a \in \reals$, the stopping time
$$
\tau_a^+ = \inf \left\lbrace t>0 \colon X_t >a \right\rbrace .
$$
It is known that (see e.g.\ Equation (16) in~\cite{lkabous-renaud_2019}), for $x \leq a$,
\begin{equation}\label{eq:two-sided}
\e_x \left[ \mathrm{e}^{-q \tau_a^+} \ind_{\{\tau_a^+<\kappa^p\}} \right] = \frac{Z_q(x,\Phi(p+q))}{Z_q(a,\Phi(p+q))} .
\end{equation}
As in \cite{renaud-zhou_2007}, we can show that
\begin{multline*}
\left( v_b(b) + \frac{1}{n} \right) \e_{b-1/n} \left[ \mathrm{e}^{-q \tau_b^+} \ind_{\{\tau_b^+<\kappa^p\}} \right] \\
\leq v_b(b) \leq \left( v_b(b) + \frac{1}{n} \right) \e_{b} \left[ \mathrm{e}^{-q \tau_{b+1/n}^+} \ind_{\{\tau_{b+1/n}^+<\kappa^p\}} \right] + o(1/n) .
\end{multline*}
The result for $x=b$ follows by taking a limit and then the result for $0 \leq x \leq b$ follows by using again the identity in~\eqref{eq:two-sided}. Finally, if $x<0$, then using~\eqref{eq:value-under-zero} we have
$$
v_b(x)=\mathrm{e}^{\Phi(p+q) x} \frac{Z_q (0,\Phi(p+q))}{Z_q^\prime (b,\Phi(p+q))}=\frac{Z_q (x,\Phi(p+q))}{Z_q^\prime (b,\Phi(p+q))} .
$$

\section{Proof of Proposition~\ref{P:optimal-barrier-level}}\label{A:proof-optimal-barrier}


Recall from~\eqref{eq:Zqprime-second-rep} that, for $x \in (0,\infty)$, we have
\begin{equation}\label{}
Z_q^\prime (x,\Phi(p+q)) = p \int_0^\infty \mathrm{e}^{-\Phi(p+q) y} W^{(q)\prime}(x+y) \mathrm{d}y .
\end{equation}
By Theorem 1.2 in \cite{loeffen-renaud_2010}, if the tail of the Lévy measure is log-convex, then $W^{(q)\prime}$ is log-convex. Using the properties of log-convex functions, as presented in \cite{roberts-varberg_1973}, we can deduce that $x \mapsto p \mathrm{e}^{-\Phi(p+q) y} W^{(q)\prime}(x+y)$ is log-convex on $(0,\infty)$, for any fixed $y \in (0,\infty)$. Then, as Riemann integrals are limits of partial sums, we have that $x \mapsto Z_q^\prime (x,\Phi(p+q))$ is also a log-convex function on $(0,\infty)$. In particular, $Z_q^\prime (\cdot,\Phi(p+q))$ is convex on $(0,\infty)$, so we can write, for some fixed $c>0$,
$$
Z_q^\prime (x,\Phi(p+q)) = Z_q^\prime (c,\Phi(p+q)) + \int_c^x Z_q^{\prime \prime -} (y,\Phi(p+q)) \mathrm{d}y,
$$
where $Z_q^{\prime \prime -} (\cdot,\Phi(p+q))$ is the left-hand derivative of $Z_q^{\prime} (\cdot,\Phi(p+q))$. Since $Z_q^{\prime \prime -} (\cdot,\Phi(p+q))$ is increasing and $\lim_{x\rightarrow\infty} Z_q^{\prime} (x,\Phi(p+q))=\infty$, we have that the function $Z_q^{\prime} (\cdot,\Phi(p+q))$ is ultimately strictly increasing. This proves that $b_p^\ast$ is well-defined.
%
%

It is known that $W^{(q)\prime}$ is strictly increasing on $(b_\infty^\ast,\infty)$; see \cite{loeffen-renaud_2010}. Then, using together the representations of $Z_q^{\prime} (x,\Phi(p+q))$ given in~\eqref{eq:Zqprime-def} and~\eqref{eq:Zqprime-second-rep}, we obtain
\begin{multline*}
Z_q^{\prime \prime} (x,\Phi(p+q))=\Phi(p+q) p \int_0^\infty \mathrm{e}^{-\Phi(p+q) y} W^{(q)\prime}(x+y) \mathrm{d}y - p W^{(q)\prime}(x) \\
> p W^{(q)\prime}(x) \int_0^\infty \Phi(p+q) \mathrm{e}^{-\Phi(p+q) y} \mathrm{d}y - p W^{(q)\prime}(x) = 0 ,
\end{multline*}
for all $x > b_\infty^\ast$. In other words, $x \mapsto Z_q^\prime (x,\Phi(p+q))$ is strictly increasing on $(b_\infty^\ast,\infty)$. Consequently, $b_p^\ast \leq b_\infty^\ast$.


The rest of the proof is similar to Lemma 3 in \cite{kyprianou-et-al_2012}, where a function closely related to one of the representations of $Z_q^\prime (x,\Phi(p+q))$ appears. For simplicity, set $g(x)=Z_q^\prime (x,\Phi(p+q))$. Using~\eqref{eq:Zqprime-def}, we can write, for $x>0$,
$$
g^\prime(x) = \Phi(p+q) \left( g(x) - \frac{p}{\Phi(p+q)} W^{(q)\prime}(x) \right) .
$$
It follows that $g^\prime(x)>0$ (resp.\ $g^\prime(x)<0$) if and only if $\displaystyle{g(x)>\frac{p}{\Phi(p+q)} W^{(q)\prime}(x)}$ (resp.\ $\displaystyle{g(x)<\frac{p}{\Phi(p+q)} W^{(q)\prime}(x)}$). This means $\displaystyle{g(b)>\frac{p}{\Phi(p+q)} W^{(q)\prime}(b)}$ for $b<b_p^\ast$ and $\displaystyle{g(b)<\frac{p}{\Phi(p+q)} W^{(q)\prime}(b)}$ for $b>b_p^\ast$. If $b_p^\ast>0$ then $g(b_p^\ast)=(p/\Phi(p+q)) W^{(q)\prime} (b_p^\ast)$.

We deduce that $b_p^\ast>0$ if and only if $g(0+)<(p/\Phi(p+q)) W^{(q)\prime} (0+)$, where $g(0+)=\Phi(p+q)-pW^{(q)}(0)$. Written differently, we have $b_p^\ast>0$ if and only if
$$
\Phi(p+q)-pW^{(q)}(0) < \frac{p}{\Phi(p+q)} W^{(q)\prime} (0+) .
$$
If $\sigma>0$, then $W^{(q)}(0)=0$ and $W^{(q)\prime}(0+)=2/\sigma^2$, which implies that $b_p^\ast>0$ if and only if
$$
\frac{ \left(\Phi(p+q) \right)^2}{p} < \frac{2}{\sigma^2} .
$$
If $\sigma=0$ and $\nu(0,\infty)=\infty$, then $W^{(q)\prime}(0+)=\infty$, which implies that $b_p^\ast>0$. Finally, if $\sigma=0$ and $\nu(0,\infty)<\infty$, then $W^{(q)}(0)=1/c$, where $c>0$ is the drift, and $W^{(q)\prime}(0+)=(q+\nu(0,\infty))/c^2$, which implies that $b_p^\ast>0$ if and only if
$$
\Phi(p+q) - \frac{p}{c} < \frac{p}{\Phi(p+q)} \frac{q+\nu(0,\infty)}{c^2} .
$$

\end{document}